\documentclass[11pt,reqno]{amsart}
\usepackage{amsmath,amsthm,amssymb,enumerate}
\usepackage{hyperref}
\usepackage{mathrsfs}

\usepackage{color}

\newtheorem{theorem}{Theorem}[section]
\newtheorem{lemma}[theorem]{Lemma}

\newtheorem{corollary}[theorem]{Corollary}

\theoremstyle{definition}
\newtheorem{assumption}{Assumption}[section]
\newtheorem{definition}{Definition}[section]

\theoremstyle{remark}
\newtheorem{remark}{Remark}[section]

\newcommand\bR{\mathbb{R}}
\newcommand\bP{\mathbb{P}}

\newcommand*{\beq}{\begin{equation}}
\newcommand*{\eeq}{\end{equation}}
\newcommand{\E}{\mathbb{E}}

\newcommand{\bit}{\begin{itemize}}
\newcommand{\eit}{\end{itemize}}
\newcommand\D{\partial}
\newcommand{\supp}{\text{supp}}
\newcommand{\dist}{\text{dist}}

\begin{document}

\title[Symmetrization of exterior problems]{Symmetrization of exterior parabolic problems and probabilistic interpretation}

\author[K. Dareiotis]{Konstantinos Dareiotis
}
\address{Department of Mathematics, Uppsala University, Box 480, 
751 06 Uppsala, Sweden}
\email{konstantinos.dareiotis@math.uu.se}


\begin{abstract}
We prove a comparison theorem for the averages of the  solutions of two exterior parabolic problems,  the second being the ``symmetrization'' of the first one, by using approximation of the Schwarz symmetrization by polarizations, as it was introduced in \cite{SOL}. This comparison provides an alternative proof, based on PDEs,  of the isoperimetric inequality for the Wiener sausage, which was proved in \cite{PS}. 
\end{abstract}

\maketitle

\section{Introduction}
In the present article we prove a comparison theorem for the average in space, at any time $t$, for  the solutions of two parabolic exterior problems, the second being  the ``symmetrization'' of the first one. In order to do so,  we show that the average of the solution decreases under polarization, and since the Schwarz symmetrization  is the limit of compositions of polarizations, we carry the comparison to the limit. This technique was introduced in \cite{SOL}.  

Our result is motivated by a problem in probability theory. Namely,  the isoperimetric inequality for the Wiener sausage, which was proved in \cite{PS}. The problem is the following. If $(w_t)_{t\geq 0}$ is a Wiener process in $\mathbb{R} ^d$, one wants to minimize the expected volume of the set $\cup_{t \leq T} (w_t+A)$, for $T\geq 0$,  over ``all"  subsets $A$ of $\bR ^d$ of a given measure. It was proved in \cite{PS} that the minimizer is the ball (the result was for a more general setting, see Section 2 below).   This was proved by obtaining a similar result for random walks by using 
rearrangement inequalities of Brascamp-Lieb-Luttinger type on the sphere, which were proved in \cite{BS}, and then by  Donsker's theorem, the authors obtain the result for the Wiener process. It is known that the expected volume of the Wiener sausage up to time $t$, can be expressed as the average in $x\in \bR^d$ of the probability that a Wiener process starting from $x \in \bR^d$ hits the set $A$ by time $t$.  It is also known that this collection of probabilities, as a function of $(t,x)$,  satisfies a parabolic equation on  $(0,T) \times \bR^d\setminus A$. For properties of these hitting times and applications to the Wiener sausage we refer the reader to \cite{VDB} and references therein,  and for the case of Riemannian manifolds, we refer to \cite{GRY}.  Therefore, we provide an alternative proof of the isoperimetric inequality for the Wiener sausage, based on PDE techniques. 

Comparison results between solutions of partial differential equations  and solutions of their symmetrized counterparts, were first proved  in \cite{TAL}.  Since then, much work  has been done in this area, for elliptic and parabolic equations, and we refer the reader to \cite{TAFA}, \cite{KES}, \cite{BAN}, \cite{SOL} and references therein.  The equations under consideration at these works, are on a bounded domain, with Dirichlet or Neumann boundary conditions. Our approach is based on the techniques introduced in \cite{SOL}. 

Let us now introduce some notation that will be frequently used throughout the paper.  We denote by $\bR^d$ the Euclidean space of dimension $1 \leq d < \infty$. For $A$, $B$ subsets of $\bR^d$, we write 
$$
A+B :=\{z \in \bR^d \ | \ z= x+y, \ x \in A, \ y \in B\}, 
$$
and for $x \in \bR^d$ we write $x+A:= \{x\}+A$. The open ball of radius $\rho >0$ in $\bR^d$ will be denoted by $B_\rho$. 
Let  $x \in \bR^d$ and $A \subset \bR^d$ and let $H$ be a closed half-space. If $A$ is measurable, $|A|$ will stand for the Lebesgue measure of $A$.  We will write $\sigma_H(x)$ and $A_H$ for the reflections of $x$ and $A$ respectively, with respect to the shifted hyperplane $\D H$.  We will write $\overline{A}$ and $\underline{A}$ for the closure and the interior of $A$ respectively. 
We will use the notation  $P_HA$  for the polarization of $A$ with respect to $H$, that is
$$
P_HA:= \Big( \left( A \cup  A_H   \right) \cap H \Big) \cup \Big(A \cap A_H \Big). 
$$
For a non-negative function $u$ on $\bR^d$ we will write  $P_H u$ for the polarization of $u$ with respect to $H$, that is
$$
 P_H u (x) = \begin{cases}
      \max\{u(x), u(\sigma_H(x))\} , & \text{if} x \in H\\
       \min\{u(x), u(\sigma_H(x))\} , & \text{if} x \in H^c\\
        
        \end{cases}  
$$
We will denote by $\mathcal{H}$ the set of all half-spaces $H$ such that $0 \in H$. For positive functions $f$ and $g$ on $\bR^d$ and for $H \in \mathcal{H}$, we will write $f \lhd_H g$, if $f(x)+f(\sigma_H(x) ) \leq g(x)+g(\sigma_H(x) )$ for a.e. $x\in H$. 
For a bounded set $V \subset \bR^d$, we will denote by $V^*$  the closed,  centered ball of volume $|V|$. For a positive function $u$ on $\bR^d$ such that $|\{u> r \}|< \infty$ for all $r >0$,  we denote by $u^*$ its symmetric decreasing rearrangement.  
For an open set $D \subset \bR^d$ we  denote by  $H^1(D)$ the space of all functions in $u \in L_2(D)$ whose distributional derivatives $\D_i u:=\frac{\D}{\D x_i}u$,  $i=1,..,d$,  lie  in $L_2(D)$, equipped with the norm 
$$
\|u\|^2_{H^1} = \|u\|^2_{L_2}+ \sum_{i=1}^d\| \D_i u\|^2_{L_2}. 
$$
We will write $H_0^1(D)$ for the closure of $C^\infty_c(D)$ (the space of smooth,  compactly supported real functions on $D$) in $H^1(D)$.  
 We will write 
$\mathbb{H}^1(D)$, and $\mathbb{H}^1_0(D)$   for $L_2((0,T);H^1(D))$, and $L_2((0,T);H^1_0(D))$ respectively. Also we define 
,  $\mathscr{H}^1(D):=\mathbb{H}^1(D) \cap C([0,T];L_2(D))$ and 
$\mathscr{H}^1_0(D):=\mathbb{H}^1_0(D) \cap C([0,T];L_2(D))$.
The notation $(\cdot,\cdot)$, will be used for the inner product in $L_2(\bR^d)$. Also, the summation convention with respect to integer valued repeated indices will be in use.

The rest of the article is organized as follows. In  Section 2 we state our main results. In Section 3 we prove a version of a parabolic maximum principle, and some continuity properties of the solution map with respect to the set $A$. These tools are then used in Section 4 in order to prove the main theorems.
\section{Main results}
Let $(\Omega, \mathscr{F},  \bP)$ be a  probability space carrying a standard  Wiener process $(w_t)_{t \geq 0}$ with values in $\bR^d$,  
and let  $A$ be compact subset of $\bR^d$.  For $T \geq 0$ we let us consider the expected volume of the Wiener sausage generated by $A$, that is, the quantity $\E \left| \cup_{t \leq T} \left(w_t+A\right) \right|$.
In \cite{PS},  the following theorem is proved.
\begin{theorem}               \label{thm: wiener sausage}
For any $T \geq 0$ we have
\begin{equation}                                             \label{eq: main inequality}
\E \left| \cup_{t \leq T} \left(w_t+A^*\right) \right|\leq \E \left| \cup_{t \leq T} \left(w_t+A\right) \right|.
\end{equation}
\end{theorem}
The result in \cite{PS} is stated for open sets $A$, and  the  set $A$ is allowed to depend on time.  As it was mentioned above, this was proved by obtaining a similar inequality for random walks, using 
rearrangement inequalities of Brascamp-Lieb-Luttinger type on the sphere, which were proved in \cite{BS}, and then by using Donsker's theorem, the authors obtain the inequality for the Wiener process.  

Let us now move to our main result, and see the connection with Theorem \ref{thm: wiener sausage}. 
For a compact set $A \subset \bR^d $, and for $\psi\in L_2(\bR^d\setminus A)$,    let us denote by $\Pi(A,\psi)$ the problem 
\begin{equation} \label{eq:parabolic pde}
  \left\{ \begin{array}{ll}
         dv_t = \frac{1}{2}\Delta v_t \ dt & \mbox{in $(0,T) \times \bR^d\setminus A$};\\
         v_t(x)=1 & \mbox{on $[0,T] \times \partial A$}; \\
         v_0(x) = \psi(x) & \mbox{in $\bR^d\setminus A$} \end{array} \right. 
        \end{equation}
  
  \begin{definition}
  We will say that $u$ is a solution of the problem   $\Pi(A, \psi)$ if
\begin{itemize}
   \item[i)]$u \in     \mathscr{H}^1(\bR^d \setminus A)$, 
  \item[ii)] for each  $\phi \in C^\infty_c(\bR^d \setminus A)$, 
  $$
        (u_t,\phi)=(\psi, \phi)-\int_0^t\frac{1}{2}(\D_iu_s , \D_i \phi) \ ds, 
  $$
  for all $t\in [0,T]$
  \item[iii)]   $v-\xi \in \mathbb{H}^1_0(\bR^d\setminus A)$, for any $\xi \in H^1_0(\bR^d)$ with $\xi =1$ on a compact set $A'$,  $A \subset \underline{A'}$. 
\end{itemize}  
\end{definition}

The following is very well known.

\begin{theorem}
There exists a unique  solution of the problem $\Pi(A,\psi)$. 
\end{theorem}

Our two main results read as follows.

\begin{theorem}
 \label{lem: comparison polarization}
 Let $\psi \in L_2(\bR^d)$ with $0\leq \psi \leq 1$, $\psi=1$ on $A$. 
Let  $u, v $ be the solutions of the problems 
$\Pi(A,\psi)$ and $\Pi(P_HA,P_H\psi)$, extended to 1 on $A$ and $P_HA$ respectively. 
Then   for all $t \in [0,T]$, we have  
$
v_t \lhd_H u_t.
$
\end{theorem}
 
\begin{theorem}                                  \label{thm: main theorem 1}
Let $\psi \in  L_2(\bR^d)$ with $0 \leq \psi \leq 1$, and $\psi=1$ on $A$. Suppose that $|A|>0$.   Let $u$, $v$ be the solutions of the problems $\Pi(A, \psi)$
and $\Pi(A^*, \psi^*)$ respectively . Then for any $t \in [0,T]$ we have
\begin{equation}                              \label{eq: inequality main theorem 1}
\int_{\bR^d} v_t  \  dx \leq \int_{\bR^d} u_t  \ dx ,
\end{equation}
where $u_t$ and $v_t$ are extended to $1$ on $A$ and $A^*$ respectively. 
\end{theorem}
 It is easy to check that 
$$
\E \left|  \cup_{t \leq T} \left(w_t+A\right) \right| = \int_{\bR^d} \bP( \tau_A^x \leq t) \ dx.
$$
where
$$
\tau_A^x:=\inf\{ t \geq 0 : x+ w_t \in A \}.
$$
It is also known that the unique solution of the problem $\Pi(A,0)$ is given by 
\begin{equation}                         \label{eq: representation of solution}
u_t(x)= \bP(\tau^x_A \leq t).
\end{equation}
Consequently Theorem \ref{thm: wiener sausage} follows by Theorem \ref{thm: main theorem 1} by choosing $\psi=0$,  if $|A|>0$. If $|A|=0$ then \eqref{eq: main inequality} trivially holds. 

\begin{remark}
All of the arguments in the next sections can be repeated in exactly the same way,  if the operator $\frac{1}{2}\Delta$ is replaced by an operator of the form $L_t  u:= \D_i (a^{ij}_t \D_ju)$, 
such that for $ j,i \in \{1,...,d\}$, $a^{ij} \in L_\infty((0,T))$, and there exists a constant $\kappa>0$ such that for almost all $t \in [0,T]$, 
\begin{equation}                        \label{eq: non degeneracy}
a^{ij}_t z_iz_j \geq \kappa |z|^2,
\end{equation}
for all $z=(z_1,...,z_d)   \in \bR^d $.   Consequently one can replace $w_t$ in Theorem \ref{thm: wiener sausage} by ``non-degenerate'' stochastic integrals  of the form
$ y_t =\int_0^t \sigma_s  \ d B_s $
where $B_t$ is an $m$-dimensional Wiener process and $\sigma$ is a measurable function from $[0,T]$ to the set of $d\times m$ matrices such that $(\sigma_t \sigma^\top_t)_{i,j=1}^d$ satisfies  \eqref{eq: non degeneracy}. 
\end{remark}

\section{Auxiliary Results}
In this section we prove some tools that we will need  in order to obtain the proof our main theorems. Namely, we present a version of the parabolic maximum principle for functions that are not necessarily continuous up to the parabolic boundary. The maximum principle is the main tool used in order to show the comparison of the solution of the problem $\Pi(A, \psi)$ and its polarized version. The reason that we need this version of the maximum principle is that, $P_H A$ is not guaranteed to have any ``good'' properties, even if $\D A$ is  of class $C^ \infty$, and therefore one can not expect the solution of $\Pi(P_HA, P_H \psi)$ to be continuous up to the boundary. We also  present some continuity properties of the solution map with respect to the set $A$, so that we can then iterate Theorem \ref{lem: comparison polarization} with respect to a sequence of half-spaces and pass to the limit, in order to obtain Theorem \ref{thm: main theorem 1}.  

In this section we consider $a^{ij} 
\in L_\infty((0,T) \times \bR^d)$ for $i,j=1,...,d$, and we assume that there exists a constant $\kappa >0$ such that for any $z=(z_i,...,z_d) \in \bR^d$ we have
$$
 a^{ij}_t(x)z_i z_j \leq \kappa |z|^2,
$$
for a.e. $(t,x)\in[0,T] \times \bR^d$.  We will  denote by $K:= \max_{i,j} \|a^{ij}\|_{L_\infty}$. For an open set $Q \subset \bR^d$, let 
  $\Psi(Q)$ be the set of functions  $u \in \mathscr{H}^1(Q)$, such that for any $\phi \in C^\infty_c(Q)$
\begin{equation}                                  \label{eq: harmonic}
(u_t,\phi)=(u_0,\phi)-\int_0^t (a^{ij}_s\D_iu_s, \D_i \phi) \  ds,
\end{equation}
for all $t\in [0,T]$. 
Notice that by the De Giorgi-Moser-Nash theorem, if $u \in \Psi(Q)$, then $u \in C((0,T)\times Q)$. 

Let us also  introduce  the functions  $\alpha_r(s)$, 
$\beta_r(s)$ and $\gamma_r(s)$ on 
$\mathbb{R}$, for  $r>0$, that will be needed in the next lemma,   given by
\begin{equation}
\gamma_ r (s)= \left\{
\begin{array}{rl}
2 & \text{if } s > r \\
\frac{2s}{r} & \text{if } 0 \leq s \leq r \\
0 & \text{if } s < 0,
\end{array} \right. \nonumber
\end{equation}
\[ \beta _r (s) = \int_0^s \gamma_r (t) \ dt,
\qquad \alpha_r(s)=\int_0^s \beta (t) \ dt.\]
For all $s \in \mathbb{R}$ we have $\gamma_r(s) \to 2I_{s >0}$, 
$ \beta_r(s) \to 2s_+$ 
and $\alpha_r(s) \to (s_+)^2$ as $r \to 0$. 
Also, for all $s\in\bR$ and $r>0$,  
the following inequalities hold
\[|\gamma_r(s)|\leq 2, \ |\beta_r(s)| 
\leq 2|s|, \ |\alpha_r(s)| \leq s^2.\]

\begin{lemma}                             \label{lem: maximum principle}
Let $Q$ be a  bounded open set and let $u\in \Psi(Q)$. Suppose that  there exists  $M\in \bR$, such that $u_0(x)\leq M$ for a.e. $x \in Q$ and $\limsup_{(t,x) \to (t_0, x_0)} u_t(x) \leq M$ for any $(t_0,x_0) \in (0,T] \times \D Q$, then
$$
\sup_{t \in [0,T]} \sup_{Q} u_t(x)  \leq M. 
$$ 
\end{lemma}

\begin{proof}
Let us fix $t' \in (0,T)$,  and let $\zeta \in C^\infty_c(B_1)$ be a positive function with unit integral.  For  $\varepsilon>0$ and $\delta>0$, set $\zeta^\varepsilon(x)= \varepsilon^{-1} \zeta(x/ \varepsilon)$ and $M^\delta:=M+\delta$. For  $x \in Q^\varepsilon:= \{ x \in Q | \dist(x,  \D Q) > \varepsilon\}$, we can plug $\zeta^\varepsilon(x-\cdot)$ in \eqref{eq: harmonic} in place of $\phi$ to obtain
$$
u^\varepsilon _t(x)-M^\delta = u^\varepsilon _ {t'}(x)-M^\delta +\int_{t'}^t (a^{ij}_s \D_j u_s, \D_i\zeta^\varepsilon (x- \cdot))  \ ds, 
$$
for all $t \in [t',T]$, where $u^\varepsilon=u * \zeta^\varepsilon$. Let also $g^n \in C^\infty_c(Q)$ with $0 \leq g^n \leq 1$, $g^n=1$ on $Q^{1/n}$, $g^n=0$ on $Q\setminus Q^{1/2n}$ and choose $\varepsilon < 1/2n$. We can then multiply the equation with $g^n$, and by the chain rule we have
\begin{align}
\nonumber 
\int_Q \alpha_r((u^\varepsilon_t-M^\delta)g_n)  \ dx =& \int_Q \alpha_r((u^\varepsilon_{t'}-M^\delta)g_n)  \ dx
\\   \nonumber
-&\int_{t'}^t \int_Q (a^{ij}_s\D_j u_s)^\varepsilon \D_i (g_n \beta_r((u^\varepsilon_s-M^\delta)g_n) \ dx  ds.
\end{align}
By standard arguments (see e.g. \cite{KI}), letting $ \varepsilon \to 0$,   leads to 
\begin{align}          \label{eq:Ito away the boundary}
\nonumber 
\int_Q \alpha_r((u_t-M^\delta)g_n) \ dx =& 
\int_Q \alpha_r((u_{t'}-M^\delta)g_n) \ dx
\\  \nonumber
-&\int_{t'}^t \int_Q g_n^2a^{ij}_s\D_j u_s  \gamma_r(u_s-M^\delta) \D_i (u_s-M^\delta)  \ dx ds
\\ \nonumber 
-&\int_{t'}^t \int_Q a^{ij}_s\D_j u_s \D_i g_n \beta_r((u_s-M^\delta)g_n)  \ dx  ds \\
-&\int_{t'}^t \int_Q a^{ij}_s\D_j u_s  \gamma_r((u_s-M^\delta)g_n) (u_s-M^\delta)g_n\D_i g_n\ dx  ds
\end{align}
Let us also introduce the notation
$$
U^\delta_t=\{x \in Q | \ u_t(x) > M+\delta \}. 
$$
We claim that there exists $\rho>0$ such that 
$\dist(\overline{U^\delta _t}, \D Q) > \rho$ for any $t \in [{t'},T]$. For each $t\in [{t'},T]$, we have $\overline{U^\delta _t} \subset Q \cup \D Q$. Suppose now that there exists $z \in \overline{ U^\delta _t} \cap \D Q$. By the definition  of $U^\delta _t$ we have that 
$$
\limsup_{Q \ni x \to z} u_t(x) \geq M+\delta,
$$
 while by assumption we have that 
 $$
 \limsup_{Q \ni x \to z} u_t(x) \leq M,
 $$
  which is a contradiction, and therefore $\overline{U^\delta_t} \subset Q$ which means that $\dist(\overline{U^\delta_t}, \D Q ) >0$ (the sets are compact).  If  $\inf_{t\in [{t'},T]}\dist(\overline{U^\delta_t}, \D Q)= 0$,   we can find $(s,y) \in [{t'},T] \times \D Q$,  and a sequence $(t_n,x_n) \in [{t'},T] \times U^\delta_{t_n}$ such that $(t_n,x_n) \to (s,y)$ as $n \to \infty$. Then we have by the definition of $U^\delta_{t_n}$, 
 $$
\limsup_{ (x_n,t_n) \to (s,y)} u_{t_n}(x_n) \geq M+\delta,
$$
while by assumption again we have that 
 $$
\limsup_{ (x_n,t_n) \to (s,y)} u_{t_n}(x_n) \leq M,
$$
which is a contradiction, and therefore 
\begin{equation}                     \label{eq:away form the boundary}
\inf_{t\in [{t'},T]}\dist(\overline{U^\delta_t}, \D Q)= \theta >0.
\end{equation}
 Going back to \eqref{eq:Ito away the boundary}, for any $n > 1/\theta$, we have that for all $s \in [{t'},T]$
$$
\int_Q \D_i u_s \D_i g_n \beta_r((u_s-M^\delta)g_n)  \ dx  =\int_{U^\delta_s}  \D_i u_s \D_i g_n \beta_r((u_s-M^\delta)g_n)  \ dx =0,
$$
since $\D_ig_n=0$ on $Q^{1/n}$ and $U^\delta_s \subset Q^{1/n}$  by \eqref{eq:away form the boundary}. Similarly for the last term on the right hand side of  \eqref{eq:Ito away the boundary}. Therefore, letting  $n \to \infty$  and $r\to 0$ in \eqref{eq:Ito away the boundary} gives
\begin{align}
\nonumber 
\|(u_t-M^\delta)_+ \|^2_{L_2(Q)} &=\|(u_{t'}-M^\delta)_+ \|^2_{L_2(Q)} -\int_{t'}^t \int_Q |\D_i u_s|^2  I_{u_s>M^\delta}  \ dx ds  \\ \nonumber 
& \leq \|(u_{t'}-M^\delta)_+ \|^2_{L_2(Q)} . 
\end{align}
The above inequality holds for any $t'\in (0,T]$, and therefore by letting $t' \downarrow 0$ and using the continuity of $u$ (in $L_2(Q)$) we have 
$$
\|(u_t-M^\delta)_+ \|^2_{L_2(Q)} \leq \|(u_0-M^\delta)_+ \|^2_{L_2(Q)}  \leq 0,
$$
since $u_0 \leq M$. 
Consequently 
$$
\sup_{Q} u_t(x) \leq M+\delta,
$$
for any $t \in [0,T]$. 
Since $\delta$ was arbitrary, the lemma is proved. 
\end{proof}
We now continue with the continuity properties of the solution map. 
Lets us fix  $\xi\in L_2(\bR^d)$ and $f \in L_2((0,T) \times \bR^d)$. We will say that $u$ solves the problem   $\Pi_0(A, \xi, f)$ if
\begin{itemize}
   \item[i)]$u \in     \mathscr{H}_0^1(\bR^d \setminus A)$,  and
  \item[ii)] for each  $\phi \in C^\infty_c(\bR^d \setminus A)$, 
  $$
        (u_t,\phi)=(\xi, \phi)+\int_0^t \left((f_s, \phi)-(a^{ij}_s\D_iu_s , \D_j \phi) \right) \ ds, 
  $$
  for all $t\in [0,T]$. 
\end{itemize}  
   
   For $n \in \mathbb{N}$, let  $\xi^n \in L_2(\bR^d)$, $f^n \in L_2((0,T) \times \bR^d)$ and let $A_n \subset \bR^d$ be compact sets.
   \begin{assumption}                \label{as: decreasing convergence}
   $ $
   \begin{itemize}
   \item[(i)] $\xi^n \to \xi$ weakly in $L_2(\bR^d)$
   \item [(ii)]$ f^n \to f $ weakly in $ L_2([0,T]; L_2(\bR^d))$ 
   \item [(iii)] $A_{n+1} \subset A_n$  For each $n \in \mathbb{N}$, and $ \cap_n A_n=A$. 
   \end{itemize}
   \end{assumption}
    
   \begin{lemma}                      \label{lem:convergence}
Suppose Assumption \ref{as: decreasing convergence} holds, and let    $u^n$ and $u$ be the solutions of the problems $\Pi_0(A_n, \xi^n,f^n)$ and $\Pi_0(A, \xi,f)$ respectively .  Let us extend $u^n$ and $u$ to zero  on $A_n$ and $A$ respectively. Then
\begin{itemize}
\item[i)]$u^n \to u$ weakly in $\mathbb{H}^1_0(\bR^d)$ as $n \to \infty$, 

\item[ii)]  $u^n_t \to u_t$, weakly in $L_2(\bR^d)$ as $n \to \infty $,  for any $t \in [0,T]$. 
\end{itemize}   
\end{lemma}

\begin{proof}
Let us set $C_n = \bR^d \setminus A_n$ and $C = \bR^d \setminus A$. 
Clearly,  for (i) it suffices to show that there exists a subsequence with $u^{n_k}$  such that  $u^{n_k} \to u$ weakly in $\mathbb{H}^1_0(C)$. 
By standard estimates we have that there exists a constant $N$ depending only on $d, K, \kappa$, and $T$,  such that for all $n$
\begin{equation}                  \label{eq:boundedness un}
\sup_{t\leq T} \|u^n_t\|_{L_2(C_n)}^2+\int_0^T \|u^n_t\|^2_{H^1_0(C_n)}  \ dt \leq N( \|\xi^n\|^2_{L_2(C_n)}+ \int_0^T \|f^n_t\|^2_{L_2(C_n)} ).
\end{equation}
Since $u^n$ are zero on $A_n$, we can replace $C_n$ by $C$ in the above  inequality, to obtain that there exists a subsequence  $(u^{n_k})_{k=1}^\infty \subset \mathbb{H}^1_0(C)$, and a function $v \in  \mathbb{H}^1_0(C)$ such that $u^{n_k} \to v$ weakly in $\mathbb{H}^1_0(C)$.

For  $\phi \in C^\infty_c(\bR^d\setminus A)$ we have that for all $k$ large enough  $\text{supp} (\phi) \subset C_{n_k}$. Also,   $u^{n_k}$ solves 
$\Pi_0(A_{n_k}, \xi^{n_k},f^{n_k})$, and therefore
\begin{equation}                                     \label{eq:before limit}
    (u^{n_k}_t,\phi)=(\xi^{n_k}, \phi)+\int_0^t \left( (f^{n_k}_s, \phi)-(a^{ij}_s\D_iu^{n_k}_s , \D_j \phi) \right) \ ds   \ \ \text{for all $t \in [0,T]$}.
  \end{equation}
which by letting $k \to \infty$ gives 
\begin{equation}                                  \label{eq:after limit}
(v_t,\phi)=(\xi, \phi)+\int_0^t\left( (f_s, \phi)-(a^{ij}_s\D_iv_s , \D_j \phi) \right) \ ds   \ \ \text{for a.e. $t \in [0,T]$}, 
\end{equation}
which also holds for any $\phi \in H^1_0(C)$, since $C^\infty_c(C)$ is dense in the later. Hence  $v$   belongs to the space $ \mathscr{H}^1_0(D)$ (by Theorem 2.16 in \cite{K} for example), and  is a solution of  
$\Pi_0(A, \xi,f)$. By the uniqueness of the solution we get $u=v$ (as elements of $ \mathscr{H}^1_0(C)$), and this proves (i). 

Let us fix $t \in [0,T]$. It suffices to show that there exists a subsequence $u^{n_k}_t$ such that $u^{n_k}_t \to u_t$ weakly in $L_2(C)$ as $k \to \infty$.  Notice that  by \eqref{eq:boundedness un}, there exists a   subsequence $u^{n_k}_t$ which  converges weakly to some $v' \in L_2(C)$. Again, for $\phi\in C_c^\infty(C)$ and $k$ large enough, we have that \eqref{eq:before limit} holds.  As $k \to \infty$, the right hand side of \eqref{eq:before limit} converges to the right hand side of \eqref{eq:after limit}  (for our fixed $t \in [0,T]$), which is equal to $(u_t,\phi)$, while the left hand side of \eqref{eq:after limit} converges to $(v', \phi)$. Hence, $v'=u_t$ on $C$, and since $u^{n_k}_t$ converges weakly in $L_2(C)$ to $v'$,  the lemma is proved. 
\end{proof}

\begin{corollary}                                    \label{cor: corollary 1}
Suppose that (i) and (iii) from Assumption \ref{as: decreasing convergence} hold, and  let $u^n$ and $u$ be the solutions of the problems $\Pi(A_n,\psi^n)$ and $\Pi(A,\psi)$. Set $u^n=1$ and $u=1$ on $A_n$ and $A$ respectively. Then for each $t$, $u^n_t \to u_t$ weakly in $L_2(\bR^d)$ as $n \to \infty$.
\end{corollary}

\begin{proof}
Let $g \in C^\infty_c (\bR^d)$ with $g=1$ on a compact set $B$ such that $A_0 \subset \underline{B}$. Then $u^n-g$ and $u-g$ solve the problems $\Pi_0(A_n,\psi^n-g,-\frac{1}{2}\Delta g)$ and $\Pi_0(A,\psi-g,-\frac{1}{2}\Delta g)$ and the result follows by Lemma \ref{lem:convergence}. 
\end{proof}

For two compact subsets of $\bR^d$,  $A_1$ and $A_2$, we denote by $d(A_1,A_2)$ the Hausdorff distance, that is 
$$
d(A_1,A_2)= \inf \left\{ \rho \geq 0  \ | \ A_1 \subset (A_2+ \overline{B}_\rho), \ A_2 \subset (A_1+ \overline{B}_\rho) \right\}.
$$
In Lemma \ref{lem: convergence 2} below we will need the following:
\begin{remark}                           \label{rem: stability of $A$} 
Let $A \subset \bR^d$ be compact such that $\bR^d \setminus A$ is a Carath\'eory set (i.e. $\D (\bR^d \setminus A) = \D \overline{(\bR^d \setminus A})$). If $u \in H^1(\bR^d)$ and $u=0$ a.e. on $A$, then $u \in H^1_0(\bR^d \setminus A)$. To see this, suppose first that $\supp(u) \subset B_R$,  where $R$ is large enough, so that $A \subset R$. It follows that $B_R \setminus A$ is a Carath\'eodory set,  and by Theorem 7.3 (ii),  page 436 in \cite{SHA}, if $u \in H^1_0(B_R)$,  and $u=0$ a.e. on $A$, then $u \in H^1_0(B_R \setminus A)$, and therefore $u \in  H^1_0(\bR^d \setminus A)$. 
For general $u$ we can take $\zeta \in C^\infty_c(\bR^d)$, such that $0\leq \zeta\leq1$ and $\zeta(x)=1$ for $|x| \leq 1$, and set 
$\zeta^n(x)=\zeta(x/n)$. Then by the previous discussion $\zeta^n u \in H^1_0(\bR^d\setminus A)$ and since $\zeta^n u \to u$ in $H^1(\bR^d\setminus A)$ we get that $u \in H^1_0(\bR^d\setminus A)$. 
\end{remark}

\begin{assumption}                           \label{as: Hausdorff}
  $ $
   \begin{itemize}
   \item[(i)] $\xi^n \to \xi$ weakly in $L_2(\bR^d)$
   \item [(ii)]$ f^n \to f $ weakly in $ L_2([0,T]; L_2(\bR^d))$ 
   \item [(iii)] $d(A,A_n) \to 0$,  $ |A \setminus A_n|\to 0$, as $n \to \infty$, and $\bR^d \setminus A$ is a Carath\'eodory set. 
   \end{itemize}

\end{assumption}

\begin{lemma}                           \label{lem: convergence 2}
Suppose Assumption \ref{as: Hausdorff} holds, and let   $u^n$ and $u$ be the solutions of the problems $\Pi_0(A_n, \xi^n,f^n)$ and $\Pi_0(A, \xi,f)$.  Let us extend $u^n$ and $u$ to $0$  on $A_n$ and $A$ respectively.   Then

\begin{itemize}
\item[i)] $u^n \to u$ weakly  in $\mathbb{H}^1_0(\bR^d)$,
\item[ii)] $u^n_t \to u_t$ weakly in $L_2(\bR^d)$,  as $n \to \infty$, and for any $t \in [0,T]$.
\end{itemize}
\end{lemma}

\begin{proof}
As in the proof of Lemma \ref{lem:convergence}  it suffices to find  a subsequences such that the  corresponding convergences take place. By standard estimates, there exists a constant $N$ depending only on $d, \kappa, T$ and $K$,  such that for all $n \in \mathbb{N}$
\begin{equation}                  \label{eq:boundedness un 2}
\sup_{t\leq T} \|u^n_t\|_{L_2(\bR^d)}^2+\int_0^T \|u^n_t\|^2_{H^1_0(\bR^d)}  \ dt \leq N( \|\xi^n\|^2_{L_2(\bR^d)}+ \int_0^T \|f^n_t\|^2_{L_2(\bR^d)} ).
\end{equation}
Therefore, there exists a subsequence  $(u^{n_k})_{k=1}^\infty \subset \mathbb{H}^1_0(\bR^d)$, and a function $v \in  \mathbb{H}^1_0(\bR^d)$ such that $u^{n_k} \to v$ weakly in $\mathbb{H}^1_0(\bR^d)$.

For  $\phi \in C^\infty_c(\bR^d\setminus A)$, since $d(A,A_n) \to 0$ as $n \to \infty$,  we have that for all $k$ large enough  $\text{supp} (\phi) \subset \bR^d \setminus A_{n_k}$. Also,   $u^{n_k}$ solves 
$\Pi_0(A_{n_k}, \xi^{n_k},f^{n_k})$, and therefore
\begin{equation}                                     \label{eq:before limit 2}
    (u^{n_k}_t,\phi)=(\xi^{n_k}, \phi)+\int_0^t \left( (f^{n_k}_s, \phi)-(a^{ij}_s\D_iu^{n_k}_s , \D_j \phi) \right)  \ ds,    \ \ \text{for all $t \in [0,T]$}.
  \end{equation}
which by letting $k \to \infty$ gives 
\begin{equation}                                  \label{eq:after limit 2}
(v_t,\phi)=(\xi, \phi)+\int_0^t\left( (f_s, \phi)-(a^{ij}_s \D_iv_s , \D_j \phi) \right) \ ds   \ \ \text{for a.e. $t \in [0,T]$}, 
\end{equation}
Notice that for $\phi \in L_\infty(A)$, $\psi   \in L_\infty((0,T)$, 
\begin{align}
\nonumber 
\Big| \int_0^T \int_A v_t \phi \psi_t \ dx dt \Big|&= \lim_{k \to \infty} \Big|\int_0^T \int_{A\setminus A_{n_k}}  u^{n_k}_t \phi \psi_t \  dx dt\Big| \\
\nonumber 
& \leq T \|\phi\|_{L_\infty(A)} \|\psi\|_{L_\infty((0,T))} \sup_{t \leq T} \|u^{n_k}_t\|_{L_2(\bR^d)} |A\setminus A_{n_k}|
\\
\nonumber 
& \to 0,  \ \text{as} \ k \to \infty,
\end{align}
by assumption and \eqref{eq:boundedness un 2}. Consequently for almost all $t \in (0,T)$, $v_t =0$ for a.e. $x \in A$. By virtue of Remark \ref{rem: stability of $A$}, we have that $v \in \mathbb{H}^1_0(\bR^d \setminus A)$, which combined with \eqref{eq:after limit 2} implies that $v \in \mathscr{H}^1_0(\bR^d \setminus A)$ and is the unique solution of the problem $\Pi _0(A,\xi,f)$.  This proves (i). 

Let us fix $t\in [0,T]$. By \eqref{eq:boundedness un 2}  there exists a subsequence $u^{n_k}_t$  that converges weakly to some $v' \in L_2(\bR^d)$. Again, for $\phi\in C_c^\infty(C)$ and $k$ large enough, we have that \eqref{eq:before limit 2} holds.  As $k \to \infty$, the right hand side of \eqref{eq:before limit 2} converges to the right hand side of \eqref{eq:after limit 2}, which is equal to $(u_t,\phi)$, while for our fixed $t$, the left hand side of \eqref{eq:after limit} converges to $(v', \phi)$.  Hence, $v'=u_t$ on $\bR^d\setminus A$. Also if $\phi \in L_\infty(A)$ 
$$
\int_A v'  \phi \ dx = \lim_{k \to \infty} \int_A u^{n_k}_t  \phi \ dx \leq \sup\|u^n_t \|_{L_2(\bR^d)} \|\phi\|_{L_\infty(A)} |A\setminus A_{n_{k}}| \to 0,
$$
as $k\to \infty$. Therefore $v'=0=u_t$ on $A$.
This shows that $v' =u_t$ on $\bR^d$ and the lemma is proved. 
\end{proof}

As with Lemma \ref{lem:convergence}, we have the following corollary, whose proof is similar to the one of  Corollary \ref{cor: corollary 1}. 

\begin{corollary}                                    \label{cor: corollary 2}
Suppose that (i) and (iii) from  Assumption \ref{as: Hausdorff} hold and  let $u^n$ and $u$ be the solutions of the problems $\Pi(A_n,\psi^n)$ and $\Pi(A,\psi)$. Set $u^n=1$ and $u=1$ on $A_n$ and $A$ respectively. Then for each $t$, $u^n_t \to u_t$ weakly in $L_2(\bR^d)$ as $n \to \infty$.
\end{corollary}

\section{Proofs of Theorems \ref{lem: comparison polarization} and \ref{thm: main theorem 1} }

\begin{proof}{\emph{of Theorem \ref{lem: comparison polarization}. }}
 Let us assume for now that $\bR^d \setminus A$ has smooth boundary, $\psi$ is compactly supported and smooth. It follows under these extra conditions  that $u \in C^\infty([0,T] \times \overline{\bR^d \setminus A})$. Also, by the De Giorgi-Moser-Nash theorem $v$ is continuous in $(0,T) \times (\bR^d \setminus P_HA)$. 

First notice that  $0\leq u,v \leq 1$.
Let us extend $u=1$ and $v=1$ on $A$ and $P_HA$ respectively so that they are defined on the whole $\bR^d$, and 
for a function $f$ let us use the notation $\overline{f}(x):=f(\sigma_H(x))$. Clearly it suffices to show that for each $t \in (0,T]$
$$
w_t:=v_t+\overline{v}_t - u_t-\overline{u}_t \leq 0, \ \text{for a.e. $x\in H^c$}.
$$
Suppose that the opposite holds, that is, 
$$
\sup_{(0,T]}\sup_{H^c} w_t(x)= \sup_{(0,T]}\sup_{\bR^d } w_t(x)=: \alpha >0.
$$
Then we have that 
\begin{equation}                \label{eq: sup attained}
\sup_{(0,T]} \sup_{\Gamma_i}  w_t(x)= \alpha,
\end{equation}
for some $i \in \{1,2,3,4\}$, where
\begin{align*}
\Gamma_1:=&  A\cap A_r\cap H^c ,  \ 
 \  \ \Gamma_2:=( \underline{A} \setminus A_H)\cap H^c \\
\Gamma_3:=&(\underline{A_H} \setminus A)\cap H^c , \ 
\Gamma_4:= H^c\setminus (A\cup A_H) .
\end{align*} 
(Notice that the boundaries of $A$ and $A_H$ are of measure zero, since they are smooth). On $\Gamma_1$, by definition $w_t =0$ for any $t \in [0,T]$, and therefore \eqref{eq: sup attained} holds for some $i \in \{2,3,4\}$. Suppose it holds for $i=2$.  
 Since the initial conditions are compactly supported, we can find an open  rectangle $R$ with $A\cup A_H \subset R$,  such that 
\begin{equation}                                 \label{eq: small on R}
\sup_{(0,T) \times \overline{R^c} } \max\{u_t(x) ,v_t(x) \} \leq \alpha /10 .
\end{equation}
Since $w_t=v_t-\overline{u}_t=: \hat{w}_t$ on $\Gamma_2$ we have 
\begin{equation}               \label{eq: beginning of contradiction}
\sup_{(0,T]} \sup_{\Theta} \hat{w}_t \geq \alpha,
\end{equation}
where $\Theta= (H^c \setminus A_H)\cap R$. 
Since
\begin{itemize}
\item[(i)] $
\limsup_{(0,T)\times \Theta \ni (t,x) \to (t_0,x_0)} \hat{w_t} \leq 0,
$
for any $(t_0,x_0) \in (0,T) \times \D A_H$,
\item[(ii)]  $P_H\psi- \overline{\psi} \leq 0$   on $H^c$, 
\item[(iii)] inequality \eqref{eq: small on R} holds,
\end{itemize}
we obtain by virtue of Lemma \ref{lem: maximum principle} that for any $\varepsilon>0$,  there exists $(t_0, x_0) \in (0,T)\times \D H$ (in fact $x_0 \in \D H \setminus A_H$ due to (i)  above) such that 
$$
\limsup_{(0,T)\times \Theta \ni (t,x) \to (t_0,x_0)} \hat{w_t} \geq \alpha - \varepsilon.
$$
Notice that $\hat{w}$ is continuous at $(t_0,x_0)$ and therefore 
$\hat{w}_{t_0}(x_0) \geq \alpha- \varepsilon$. This implies that 
$$
w_{t_0}(x_0) =2\hat{w}_{t_0}(x_0) \geq 2( \alpha- \varepsilon)=2 \sup_{(0,T]} \sup_{\bR^d} w_t- 2 \varepsilon,
$$
which is a contradiction for $\varepsilon$ small enough . If  \eqref{eq: sup attained} holds for $i=3$
then in the same way we have that 
\begin{equation}   \label{eq: beginning of contradiction 2}
\sup_{(0,T]} \sup_{\Theta '} \tilde{w}_t \geq \alpha,
\end{equation}
where $\tilde{w}_t=v_t-u_t$ and $\Theta'=(H^c \setminus A)\cap R$.  This inequality leads to a similar contradiction. 

Finally let us assume that \eqref{eq: sup attained} holds for $i=4$. 
In particular then we have
$$
\sup_{(0,T]} \sup_{G} w_t \geq \alpha, \ \text{where}  \ G:=R\setminus (A\cup A_H).
$$
By virtue of \eqref{eq: small on R},  and since $P_H \psi  \lhd_H \psi$,   Lemma \ref{lem: maximum principle} implies that for any $\varepsilon>0$,  there exists 
$(t_0,x_0) \in (0,T] \times \D(A\cup A_H)$ such that 
$$
\limsup_{(0,T) \times G \ni(t,x) \to (t_0,x_0)} w_t(x) \geq \alpha- \varepsilon. 
$$
Notice that $x_0 \in \D A \cap A_H^c$ or $x_0 \in \D A _H \cap A^c$, because if it belongs to $\D A \cap \D A_H$ then the $\limsup$ above is less than or equal to zero. Let us consider the first case. We can assume further that $(t_0, x_0) \in \D A \cap A_H^c \cap H^c$ because of symmetry. Let $(t_n,x_n) \in G$ be a sequence converging to $(t_0,x_0)$ such that $w_{t_n}(x_n) \to \alpha- \varepsilon$. For all $n\in \mathbb{N}$ sufficiently large,  we have
$w_{t_n}(x_n) \geq \alpha- 2\varepsilon$ and $u_{t_n} (x_n) \geq 1- \varepsilon$, the last by the continuity of $u$ up to the parabolic  boundary.  Then we have for all $n$ large
$$
\alpha - 2\varepsilon \leq w_{t_n}(x_n) \leq v_{t_n}(x_n) +1-(1-\varepsilon)-\overline{u}_{t_n}(x_n). 
$$
This now implies \eqref{eq: beginning of contradiction} which we showed leads to a contradiction. For the second case, we can assume again that $(t_0, x_0) \in \D A_H \cap A^c \cap H^c$. This in the same manner leads to \eqref{eq: beginning of contradiction 2}, which also leads to contradiction.

 For general $A$ and $\psi$,  let $A_n$ be a sequence of compact sets such that  for $n \in \mathbb{N}$, 
 $\bR^d\setminus A_n$ has smooth boundary,  $A \subset A_{n+1} \subset A_n$,  and $A = \cap_nA_n$ (see e.g. page 60 in \cite{CHI}) .   Let $0 \leq \psi^n \leq 1$ be  smooth with compact support such that $\psi^n=1$ on $A_n$, and  $\|\psi^n- \psi \|_{L_2(\bR^d)} \to 0$ as $n \to 0$. Then we also have that $\|P_H\psi^n- P_H\psi\|_{L_2(\bR^d)} \to 0$ as $n \to 0$,   $P_HA \subset P_HA_{n+1} \subset P_HA_n$ for any $n \in \mathbb{N}$, and $P_HA = \cap_n P_HA_n$ (see \cite{SOL}). Let $u^n$ and $v^n$ be the solutions of the problems $\Pi(A_n,\psi^n)$ and $\Pi(P_HA_n, P_H\psi_n)$ respectively. By  Lemma \ref{lem:convergence} we have that $u^n_t$ and $v^n_t$ converge to $u_t$ and $v_t$ weakly in $L_2(\bR^d)$. In particular $z^n:=(u^n_t,v^n_t)$ converges weakly to $z:=(v_t,u_t)$ in $L_2(\bR^d;\bR^2)  $. By Mazur's lemma there exists a sequence $(g_k=(g^1_k,g^2_k))_{k \in \mathbb{N}}$ of convex  combinations  of $z^n$ such that the convergence takes place strongly. Then we can find a subsequence $g_{k(l)}$, $l \in \mathbb{N}$,  where the convergence takes place for a.e. $x\in \bR^d$.  For each $l$ we have 
 $$
g^1_{k(l)}+\overline{g^1_{k(l)}} = \sum_{i\in C}c_i (v^{i}_t+\overline{v^{i}}_t)\leq \sum_{i\in C} c_i (u^{i}_t+\overline{u^{i}}_t)= g^2_{k(l)}+\overline{g^2_{k(l)}}
 $$
 where $C \subset \mathbb{N}$ is a finite set and $c_i \geq 0$, $\sum_{i \in C} c_i=1$. 
Letting  $l \to \infty$ finishes the proof. 
 
\end{proof}

\begin{proof}{\emph{of Theorem \ref{thm: main theorem 1}}.}
First, let us assume that 
$$
\int_{\bR^d} u_t(x) \ dx < \infty,
$$
or else the conclusion of the theorem is obviously  true. 
Since $|A|>0$, it follows from \cite{BU} that there exist $H_i \in \mathcal{H}$, $i \in \mathbb{N}$,  such that 
$$
\lim_{n \to \infty} \left(\|\psi ^* - \psi^n \|_{L_2(\bR^d)}+|A^*\Delta A_n|+d(A^*,A_n) \right)= 0,
$$
where 
$$
\psi^n :=  P_{H_n}... P_{H_1} \psi, \ A_n :=P_{H_n}...P_{H_1}A.
$$
Let $u^n$ be the solution of the problem $\Pi(A_n,\psi^n)$. For $t \in [0,T]$, by virtue of  Theorem \ref{lem: comparison polarization}, we have by induction
\begin{equation}              \label{eq: inequality by induction}
 \int_{\bR^d} u^n_t \ dx \leq \int_{\bR^d} u_t \ dx,
\end{equation}
for all $n \geq 0$.  By Lemma \ref{lem: convergence 2} ($|A|>0 $ and therefore $\bR^d\setminus A^*$ is obviously a Carath\'eodory set) we have that $u^n_t \to v_t$ weakly in $L_2(\bR^d)$ as $n \to \infty$.
Hence we can find a sequence of convex combination that converges strongly, and a subsequence of it, let us call it $(v^n)_{n=1}^\infty$, such that $v^n \to v_t$ for a.e. $x \in \bR^d$. 
Since for each $n$, $v^n$ is convex combination of elements from $(u^n_t)_{n=1}^\infty$, we have by \eqref{eq: inequality by induction}
$$
\int_{\bR^d} v^n \ dx \leq \int_{\bR^d} u_t \ dx,
$$
which combined with  Fatou's lemma brings the proof to an end. 
\end{proof}

\section*{Acknowledgements}
The author would like to thank Takis Konstantopoulos and Tomas Juskevicius for the useful discussions.



\end{document}